\documentclass[11pt]{amsart}
\usepackage[left=1in,right=1in,top=1in,bottom=1in]{geometry}

\usepackage[all,cmtip]{xy}
\usepackage{amssymb}
\usepackage[english]{babel}
\usepackage{mathrsfs}

\newtheorem{theorem}{Theorem}[section]
\newtheorem{lemma}[theorem]{Lemma}
\newtheorem{proposition}[theorem]{Proposition}
\newtheorem{corollary}[theorem]{Corollary}
\newtheorem{fact}[theorem]{Fact}

\theoremstyle{definition}
\newtheorem{definition}[theorem]{Definition}

\newtheorem{example}[theorem]{Example}
\newtheorem{problem}[theorem]{Problem}

\newtheorem{remark}[theorem]{Remark}

\newtheorem{question}[theorem]{Question}

\newtheorem{claim}{Claim}

\def\Cl{\mathrm{cl}}

\def\B{\mathcal{B}}
\def\A{\mathcal{A}}
\def\N{\mathbb{N}}
\def\R{\mathbb{R}}
\def\T{\mathbb{T}}
\def\V{\mathcal{V}}
\def\C{\mathfrak{c}}
\def\sp{sequentially pseudocompact}
\def\ssp{selectively sequentially pseudocompact}

\def\strong{selectively pseudocompact}
\def\strongness{selective pseudocompactness}
\def\SSP{selectively (sequentially) pseudocompact}
\def\SSPness{selective (sequential) pseudocompactness}

\def\gp#1{\langle{#1}\rangle}

\begin{document}
\title{Selective sequential pseudocompactness}
\author[A. Dorantes-Aldama]{Alejandro Dorantes-Aldama}
\address{Department of Mathematics, Faculty of Science, Ehime University,
  Matsuyama 790-8577, Japan}
\email{alejandro\_dorantes@ciencias.unam.mx}
\thanks{The first listed author was supported by CONACyT of M\'exico: Estancias Posdoctorales al Extranjero propuesta No. 263464}

\author[D. Shakhmatov]{Dmitri Shakhmatov}
\address{Division of Mathematics, Physics and Earth Sciences\\
Graduate School of Science and Engineering\\
Ehime University, Matsuyama 790-8577, Japan}
\email{dmitri.shakhmatov@ehime-u.ac.jp}
\thanks{The second listed author was partially supported by the Grant-in-Aid for Scientific Research~(C) No.~26400091 by the Japan Society for the Promotion of Science (JSPS)}

\begin{abstract} 
We say that a topological space $X$ is {\em \ssp} 
if for every family $\{U_n:n\in\N\}$ of non-empty open subsets of $X,$
one can 
choose 
a point $x_n\in U_n$ for every 
$n\in\N$
in such a way that
the sequence 
$\{x_n:n\in \N\}$ 
has a convergent subsequence. We show that 
the class of \ssp\ spaces is closed under taking arbitrary products and continuous images, contains the class of all dyadic spaces and forms a proper subclass of the class of strongly pseudocompact spaces introduced recently by Garc\'ia-Ferreira and Ortiz-Castillo.
We 
investigate basic properties of this new class and its relations with  known compactness properties.
We prove that every $\omega$-bounded (=the closure of which countable set is compact) group is \ssp, while compact spaces need not be \ssp.
Finally, we construct \ssp\ group topologies on both the free group
and the free Abelian group  with continuum-many generators.
\end{abstract}

\dedicatory{Dedicated to Professor Filippo Cammaroto on the occasion of his 65th anniversary}



\maketitle

{\em All topological spaces are considered to be Tychonoff and all topological groups Hausdorff.\/}

The symbol $\N$ denotes the set of natural numbers, $\R$ denotes the set of real numbers, $\T$ denotes
the circle group $\{e^{i\theta} :\theta \in \R\}\subseteq \R^2.$ 

If $f:X\to Y$ is a function and $A$ is a subset of $Y$, then we let $f^{\leftarrow }(A)=\{x\in X:f(x)\in A\}$. 

For
a subset $A$ of a topological space $X$, we use cl$_X(A)$ to denote the closure of $A$ in $X$.

Let us recall two well-known 
compactness-type properties.
\begin{definition}\label{definition_1} A topological space $X$ is called:
\begin{itemize}
\item[(i)] {\em sequentially compact} if every sequence in $X$ has a convergent subsequence;
\item[(ii)] {\em pseudocompact\/} if every continuous real-valued function defined on $X$ is bounded;
\end{itemize}
\end{definition}

The notion of sequential compactness is well known; see, for example, \cite[Section 3.10]{E}.  
The notion of pseudocompactness is due to Hewitt \cite{H}.

\begin{definition}
We say that an indexed sequence $\{U_n:n\in\N\}$ is: 
\begin{itemize}
\item
{\em point-finite\/} provided that the set $\{n\in\N:x\in U_n\}$ 
is finite for each point $x\in X$;
\item
{\em locally finite\/} provided that each point $x\in X$ has an open neighborhood $W$ such that the set $N(W)=\{n\in\N: W\cap U_n\not=\emptyset\}$ is finite.
\end{itemize}
\end{definition}

Clearly, every locally finite sequence is point-finite.

The following fact is well known.

\begin{fact}
\label{psc:characterization}
For a topological space $X$, the following conditions are equivalent:
\begin{itemize}
\item[(i)] $X$ is pseudocompact;
\item[(ii)] every sequence of non-empty open subsets of $X$ is not locally finite;
\item[(ii)] every sequence of pairwise disjoint non-empty open subsets of $X$ is not locally finite.
\end{itemize}
\end{fact}

\section{Disjoint refinement principle}

The next lemma is probably known. We include its proof here only for the reader's convenience.
   
\begin{lemma}
\label{disjoint:sequence}
Let $\mathscr{U}=\{U_n:n\in\N\}$ be a point-finite family of non-empty open subsets of a topological space $X$
which is not locally finite.
Then there exist a sequence $\mathscr{V}=\{V_k:k\in\N\}$ of pairwise disjoint non-empty open subsets of $X$
and a strictly increasing sequence $\{n_k:k\in\N\}\subseteq \N$
such that $V_k\subseteq U_{n_k}$ for all $k\in\N$. 
\end{lemma}
\begin{proof}
Since $\mathscr{U}$ is not locally finite,
there exists $x\in X$ 
such that 
\begin{equation}
\label{star2}
\text{the set }
N(W)=\{n\in\N: W\cap U_n\not=\emptyset\}
\text{ is infinite for every open neighborhood }
W
\text{ of }x.
\end{equation} 
Since $\mathscr{U}$ is point-finite, we can fix $n_{-1}\in \N$
such that
\begin{equation}
\label{star}
x\not\in U_n
\text{ for every }
n\in\N
\text{ satisfying }
n\ge n_{-1}. 
\end{equation}

By induction on $k\in\N$ we shall 
select $n_k\in\N$ 
and a non-empty open subset $V_k$ of $X$ 
satisfying the following conditions:
\begin{itemize}
\item[(i$_k$)] 
$n_{k-1}<n_k$,
\item[(ii$_k$)] 
$\Cl_X(V_k)\subseteq U_{n_k}$,
\item[(iii$_k$)]  if $k\ge 1$, then $V_k\cap V_j=\emptyset$
whenever $j\in\N$ and $j<k$.
\end{itemize}

We let $n_0=n_{-1}+1$ and use regularity of $X$ to fix a non-empty open 
subset $V_0$ of $U_{n_0}$ satisfying (ii$_0$). Condition (i$_0$) is clear and condition 
(iii$_0$) is vacuous.

Suppose now that $k\in\N$, $k\ge 1$ and $n_{k-1}\in \N$
and a non-empty open subset $V_{k-1}$ of $X$ satisfying 
(i$_{k-1}$)--(iii$_{k-1}$) have already been chosen.

Let $j\in\N$ and $j<k$.
Since 
(i$_l$) holds for all $l\in\N$ with $l\le j$, we have $n_{-1}<n_0<\dots<n_j$.
Combining this with 
\eqref{star}, we conclude that
$x\not\in U_{n_j}$. Finally, from this and (ii$_j$), we get
$x\not\in \Cl_X(V_j)$.
Since this holds for every $j\in\N$ with $j<k$,   
it follows that 
 $W=X\setminus \bigcup_{j<k} \Cl_X(V_j)$ is an open neighborhood 
 of $x$. Since the set $N(W)$ is infinite by 
 \eqref{star2},
 we can choose $n_k\in N(W)$ satisfying $n_{k-1}<n_k$.
Then $W\cap U_{n_k}\not=\emptyset$ by the definition of $N(W)$,
so we can choose a point $z\in W\cap U_{n_k}$ and its open neighborhood $V_k$ such that $\Cl_X(V_k)\subseteq U_{n_k}$.
Clearly, conditions (i$_{k}$)--(iii$_{k}$) are satisfied.
The inductive step has been completed.

The sequence $\{n_k:k\in\N\}$ is strictly increasing, as (i$_k$) holds for every $k\in\N$.
Since (iii$_k$) holds for all $k\in\N$, we conclude that
$\mathscr{V}=\{V_k:k\in\N\}$ is a sequence of non-empty pairwise disjoint 
open subsets of $X$. 
Finally, $V_k\subseteq U_{n_k}$  by (ii$_k$).
\end{proof}

\begin{definition}
Let $X$ be a topological space.
\begin{itemize}
\item[(i)] $\mathbb{U}_X$ denotes the family of all sequences $U=\{U(n):n\in\N\}$ of non-empty open subsets of $X$.
\item[(ii)] We say that $U\in \mathbb{U}_X$ is {\em pairwise disjoint\/}
if $U(n)\cap U(m)=\emptyset$ whenever $m,n\in\N$ and $m\not=n$.
\item[(iii)] If $U\in\mathbb{U}_X$ and $s:\N\to\N$ is a monotonically increasing function, then the function $U\circ s\in\mathbb{U}_X$ shall be called a {\em subsequence\/} of $U$.
\end{itemize}
\end{definition}

\begin{proposition}
\label{truth:proposition}
Let $X$ be a pseudocompact space and let $\mathbf{P}:\mathbb{U}_X\to \{0,1\}$ be a function satisfying the following conditions:
\begin{itemize}
\item[(a)] $\mathbf{P}(U)=1$ whenever $U\in\mathbb{U}_X$ is not point-finite;
\item[(b)] if $\mathbf{P}(V)=1$ for some subsequence $V$ of 
$U\in \mathbb{U}_X$, then $\mathbf{P}(U)=1$.
\item[(c)] if $U,V\in\mathbb{U}_X$, $V(n)\subseteq U(n)$ for every $n\in\N$ and  $\mathbf{P}(V)=1$, then $\mathbf{P}(U)=1$.
\item[(d)] if $V\in\mathbb{U}_X$ is pairwise disjoint, then
$\mathbf{P}(V)=1$.
\end{itemize}
Then $\mathbf{P}(U)=1$ for all $U\in\mathbb{U}_X$.
\end{proposition}
\begin{proof}
Let $U\in\mathbb{U}_X$ be arbitrary. By (a), we may assume, without loss of generality, that $U$ is point-finite. Since $X$ is pseudocompact, $U$ cannot be locally finite by Fact~\ref{psc:characterization}.
By Lemma~\ref{disjoint:sequence},
there exist a pairwise disjoint $V\in\mathbb{U}_X$ and a monotonically increasing function $s:\N\to \N$ such that 
\begin{equation}
\label{V:subseq:of:U}
V(k)\subseteq U(s(k))=(U\circ s)(k)
\text{ for every }
k\in\N.
\end{equation}
Since $V\in\mathbb{U}_X$ is pairwise disjoint,
$\mathbf{P}(V)=1$ by (d).
From this, \eqref{V:subseq:of:U}
and
(c), we get $\mathbf{P}(U\circ s)=1$.
Since $U\circ s$ is a subsequence of $U$, (b) implies
$\mathbf{P}(U)=1$.
\end{proof}

\begin{corollary}
\label{equivalence:for:disjoint}
Let $X$ be a pseudocompact space and let $\mathbf{P}:\mathbb{U}_X\to \{0,1\}$ be a function satisfying conditions (a), (b), (c) from Proposition~\ref{truth:proposition}.
Then the following conditions are equivalent:
\begin{itemize}
\item[(i)] $\mathbf{P}(U)=1$ for every $U\in\mathbb{U}_X$;
\item[(ii)] $\mathbf{P}(V)=1$ for each pairwise disjoint  $V\in\mathbb{U}_X$.
\end{itemize}
\end{corollary}

We demonstrate the usefulness of this corollary 
on the 
following
property which is perhaps 
not so
familiar to the reader.

\begin{definition}
\label{def:sp}
A topological space $X$ is called 
{\em \sp\/} if for every family $\{U_n:n\in\N\}$ of non-empty open subsets of $X$,
there exists an infinite set $J\subseteq \N $ and a point $x\in X$ such that the
set $\{n\in J:W\cap U_n=\emptyset \}$ is finite for every open 
neighborhood $W$ of $x$.
\end{definition}

This notion 
appeared in \cite[Definition 1.4]{DPSW} under the name sequentially feebly compact. 
A formally weaker property obtained by requiring the conclusion of Definition~\ref{def:sp} to hold only for the families
 $\{U_n:n\in\N\}$ consisting of  pairwise disjoint non-empty open
subsets of $X$ was defined earlier in \cite[Definition 1.8]{AMPRT}.
It was proved in \cite[Proposition 1]{L} that 
these two versions of Definition~\ref{def:sp}
are in fact equivalent.  
This result can be easily derived from the disjoint refinement principle.

\begin{corollary}
\cite[Proposition 1]{L} 
For every topological space $X$, the following conditions are equivalent:
\begin{itemize}
\item[(i)] $X$ is \sp;
\item[(ii)] for every sequence $\{U_n:n\in\N\}$ of pairwise disjoint non-empty open subsets of $X$,
there exists an infinite set $J\subseteq \N $ and a point $x\in X$ such that the
set $\{n\in J:W\cap U_n=\emptyset \}$ is finite for every open 
neighborhood $W$ of $x$.
\end{itemize}
\end{corollary}
\begin{proof}
The implication (i)$\Rightarrow $(ii) is trivial.
Let us prove that (ii)$\Rightarrow$(i).
Note that (ii) implies that every pairwise disjoint sequence of non-empty  open subsets of $X$ is not locally finite. Therefore, $X$ is pseudocompact by Fact~\ref{psc:characterization}.

Define the function $\mathbf{P}:\mathbb{U}_X\to \{0,1\}$
as follows.
For $U\in\mathbb{U}_X$, we let $\mathbf{P}(U)=1$ if 
there exists an infinite set $J\subseteq \N $ and a point $x\in X$ such that the
set $\{n\in J:W\cap U(n)=\emptyset \}$ is finite for every open 
neighborhood $W$ of $x$, and we define 
$\mathbf{P}(U)=0$ otherwise.
  The function $\mathbf{P}$ clearly satisfies 
conditions (a), (b), (c) of  Proposition~\ref{truth:proposition}.
Now the conclusion follows from Corollary~\ref{equivalence:for:disjoint}.
\end{proof}

\section{Two ``selective'' pseudocompactness-type properties}

The (equivalent) properties from the next theorem can be considered ``selective'' properties, as they all involve a selection 
of a point from each  set of a countable sequence of non-empty open sets in such a way that the resulting  sequence satisfies  some condition agreed in advance.

\begin{theorem}
\label{equiv}
For every topological space $X$, the following conditions are equivalent:
\begin{itemize}
\item[(i)] for each sequence $\{U_n:n\in\N\}$ of  pairwise disjoint non-empty open subsets of $X$,
one can choose a point $x_n\in U_n$ for every $n\in\N$ such that
the set $\{x_n:n\in\N\}$ is not closed in $X$;
\item[(ii)] for each sequence $\{U_n:n\in\N\}$ of  pairwise disjoint non-empty open subsets of 
$X$, one can choose a point 
$x_n\in U_n$ for each $n\in\N$ 
such that
$\Cl_X(\{x_n:n\in \N\})\setminus \bigcup_{n\in\N}U_n\not=\emptyset$;
\item[(iii)] for each sequence $\{U_n:n\in\N\}$ of non-empty open subsets of $X$,
one can choose a point $x\in X$ and a point 
$x_n\in U_n$ for each $n\in\N$ 
such that  
the set $\{n\in \N:x_n\in W\}$ is infinite for every open neighborhood $W$ of $x$;
\item[(iv)] for each sequence $\{U_n:n\in\N\}$ of non-empty open subsets of $X$ 
on can choose a point $x_n\in U_n$ for every $n\in\N$,
select 
a free ultrafilter\footnote{Recall that an ultrafilter $p$ on $\N$ is called {\em free\/} if
$\bigcap p=\emptyset .$} 
 $p$ on $\N$ and find a point $x\in X$ such that  
$\{n\in \N:x_n\in W\}\in p$
for every open neighborhood $W$ of $x$.
\end{itemize}
\end{theorem}

\proof
(iv)$\Rightarrow $(ii) is clear.

(ii)$\Rightarrow$(i) Let $\{U_n:n\in\N\}$ be a sequence of non-empty pairwise disjoint open subsets of $X$.
Let $\{x_n:n\in\N\}$ be the sequence as in the conclusion of 
item (ii).
Since 
$\{x_n:n\in \N\}\subseteq \bigcup_{n\in\N}U_n$, yet
$\Cl_X(\{x_n:n\in \N\})\setminus \bigcup_{n\in\N}U_n\not=\emptyset$,
  the set $\{x_n:n\in \N\}$ is not closed in $X$. 

(i)$\Rightarrow$(iii)
First, note that the condition (i) implies that 
an infinite sequence of pairwise disjoint non-empty open subsets of $X$ cannot be locally finite. Applying Fact~\ref{psc:characterization}, we conclude that 
$X$ is pseudocompact. 

Define the function $\mathbf{P}:\mathbb{U}_X\to \{0,1\}$
as follows.
For $U\in\mathbb{U}_X$, we let $\mathbf{P}(U)=1$ if 
one can chose a point $x\in X$ and a point 
$x_n\in U(n)$ for each $n\in\N$ 
such that  
the set $\{n\in \N:x_n\in W\}$ is infinite for every open neighborhood $W$ of $x$,
and we define 
$\mathbf{P}(U)=0$ otherwise.
  The function $\mathbf{P}$ clearly satisfies 
conditions (a), (b), (c) of  Proposition~\ref{truth:proposition}.

Let 
use check the  condition (d) of this proposition.
Suppose that $V\in\mathbb{U}_X$ is pairwise disjoint.
Apply (i) to choose a point $x_k\in V(k)$ for every $k\in\N$
such that the set $\{x_k:k\in\N\}$ is not closed in $X$.
Let $x\in \Cl_X(\{x_k:k\in \N\})\setminus \{x_k:k\in \N\}$.
  One can easily check now that $\{k\in W: x_k\in W\}$ 
  is infinite for every neighborhood $W$ of $x$.

We have checked that $\mathbf{P}$ satisfies all the assumptions of Proposition~\ref{truth:proposition}, which implies that 
$\mathbf{P}(U)=1$ for all $U\in\mathbb{U}_X$. This means that 
(iii) holds.

(iii)$\Rightarrow$(iv) Let $\{U_n:n\in\N\}$ be a sequence of non-empty open subsets of $X$. 
By (iii), we
can chose a point $x\in X$ and a point 
$x_n\in U_n$ for each $n\in\N$ 
such that  
the set $N_W=\{n\in \N:x_n\in W\}$ is infinite for every open neighbourhood $W$ of $x$.
Therefore,
$$
\mathcal{B}=\{N_W: W
\text{ is an open neighborhood of }
x\}\cup\{\N\setminus F: F
\text { is a finite subset of }
\N\}
$$
is a base of a free filter on $\N.$ Let $p$ be an ultrafilter on $\N$
satisfying $\mathcal{B}\subseteq p$; such an ultrafilter exists by Zorn's lemma. 
One easily checks that $p$ satisfies (iv).
\endproof

\begin{definition}
\label{def:new:name}
A topological space $X$ satisfying any (and then all) of the equivalent conditions in Theorem~\ref{equiv} 
will be called {\em \strong\/}.
\end{definition}

The property from item (iv) of Theorem~\ref{equiv}
appeared recently in~\cite{GO,GT}
under the name ``strong pseudocompactness''.
The property from item (ii) of Theorem~\ref{equiv} appeared under the same name in
 the abstract of  \cite{GT}.
Since the term ``strongly pseudocompact'' is used in  \cite{AG,Di} to denote two different properties, Definition~\ref{def:new:name} proposes a new name for 
this property
reflecting its``selective'' nature.
The proposed new name  also matches the name of another
property introduced in Definition~\ref{def:sequential:seq:psc} below.

To the best of our knowledge, the following ``selective'' property is new. 

\begin{definition}
\label{def:sequential:seq:psc}
We shall call a topological space $X$ 
{\em \ssp} 
if for every family $\{U_n:n\in\N\}$ of non-empty open subsets of $X,$
one can 
choose 
a point $x_n\in U_n$ for every 
$n\in\N$
in such a way that
the sequence 
$\{x_n:n\in \N\}$ 
has a convergent subsequence.
\end{definition}

\begin{proposition}
\label{disjoint:condition:for:ssp}
For every topological space $X$, the following conditions are equivalent:
\begin{itemize}
\item[(i)] $X$ is \ssp;
\item[(ii)] for every sequence $\{U_n:n\in\N\}$ of pairwise disjoint non-empty open subsets of $X$, one can choose $x_n\in U_n$ for all $n\in \N$ such that the sequence $\{x_n:n\in\N\}$ has a convergent subsequence.
\end{itemize}
\end{proposition}
\begin{proof}
(i)$\Rightarrow$(ii) trivially follows from Definition~\ref{def:sequential:seq:psc}.

(ii)$\Rightarrow$(i) One easily sees that (ii) implies that every sequence pairwise disjoint non-empty open subsets of $X$ fails to be locally finite. Therefore, $X$ is pseudocompact by Fact~\ref{psc:characterization}.

Define the function $\mathbf{P}:\mathbb{U}_X\to \{0,1\}$
as follows.
For $U\in\mathbb{U}_X$, we let $\mathbf{P}(U)=1$ if 
one can chose a point 
$x_n\in U(n)$ for each $n\in\N$ 
such that  
the sequence $\{x_n:n\in\N\}$ has a convergent subsequence,
and we define 
$\mathbf{P}(U)=0$ otherwise.
  The function $\mathbf{P}$ clearly satisfies 
conditions (b) and (c) of  Proposition~\ref{truth:proposition}.

Let us check also condition (a).  If $U\in\mathbb{U}_X$ is not point-finite, then there exists $x\in X$
such that the set $N=\{n\in\N: x\in U(n)\}$ is infinite.
Define $x_n=x$ for all $n\in  N$ and choose $x_n\in U(n)$ arbitrarily
for each $n\in\N\setminus N$.
Now the sequence $\{x_n:n\in\N\}$ has a constant (thus, convergent) subsequence.

Finally, (ii) implies that $\mathbf{P}$ satisfies condition (d)
as well.

We have checked that $\mathbf{P}$ satisfies all the assumptions of Proposition~\ref{truth:proposition}, which implies that 
$\mathbf{P}(U)=1$ for all $U\in\mathbb{U}_X$. This means that 
(i) holds.
\end{proof}

The following diagram summarizes connections between 
two ``selective'' pseudocompactness-type properties and other compactness-like properties.

\begin{center} \medskip\hspace{1em}
\xymatrix{
& \text{compact}\ar[d]\\
\text{sequentially compact}\ar[r]\ar[d]_1 & \text{countably compact}\ar[d]_2\\
\text{\ssp}\ar[r]\ar[d]_3 & \text{\strong}\ar[d]_4\\
\text{\sp}\ar[r]& \text{pseudocompact}
}
\label{Diagram:1}
\end{center}
\begin{center}
Diagram 1.
\end{center}

\begin{example}
{\em A compact space need not be \sp\/}. 
Indeed,
the Stone-\v{C}ech compactification $\beta \omega $ of $\omega $ is not \sp\ 
%; see 
\cite[Example 2.9]{DPSW}.
Hence, none of the properties on the right side 
of Diagram 1 imply any of the properties on the left side. 
\end{example}

\begin{example}
Examples of selectively pseudocompact spaces which are not countably compact were constructed in \cite{GT}.
Therefore, arrow 2 is not reversible.
\end{example}

\begin{example}
Let $X$ be
a pseudocompact space such that all its countable subsets are closed constructed in 
\cite[Theorem 2]{S}. 
By Theorem~\ref{equiv}~(i), $X$ is
not \strong. Hence, arrow 4 is not reversible.
\end{example}

In Example~\ref{ex:5.4} below we shall show that 
arrow 1 is not reversible. Arrow 3 is not reversible either; see Example~\ref{ex:1} below.

\section{Basic properties of \SSPness}

In this section we collect basic properties of the class of \SSP\ spaces.

\begin{proposition}
\label{sequences:in:ssp:spaces}
Every infinite \ssp\ space has a non-trivial convergent sequence.
\end{proposition}
\begin{proof}
Let $X$ be an infinite \ssp\ space. Since $X$ is Hausdorff,
there exists a sequence $\{U_n:n\in\N\}$ consisting of pairwise disjoint non-empty open subsets of $X$.
By Proposition~\ref{disjoint:condition:for:ssp}, 
one can choose $x_n\in U_n$ for all $n\in \N$ such that the sequence $\{x_n:n\in\N\}$ has a convergent subsequence.
Since $U_n\cap U_m=\emptyset$ whenever $m,n\in\N$ and $m\not=n$, this subsequence is non-trivial.
\end{proof}

\begin{proposition}
\label{continuous_image}
\label{cont_image_strong_pseudo}
Let $f:X\to Y$ be a continuous function from a topological space $X$ onto a topological space $Y.$ If $X$ is \SSP, then so is $Y$. 
\end{proposition}

\proof
Let $\{V_n:n\in\N\}$ be a sequence of pairwise disjoint non-empty open subsets of $Y$.
For every $n\in\N$ define $U_n=f^{\leftarrow }(V_n)$.
Since $f$ is continuous and onto, $\{U_n:n\in\N\}$ is a sequence of pairwise disjoint non-empty open subsets of $X.$ 
Now the proof splits into two cases.

If $X$ is \ssp, one can 
choose 
a point $x_n\in U_n$ for every 
$n\in\N$
such that
the sequence 
$\{x_n:n\in \N\}$ 
has a convergent subsequence.
Clearly $f(x_n)\in V_n$ for every $n\in\N.$
Since $f$ is continuous, the sequence $\{f(x_n):n\in \N\}$ has a convergent subsequence in $Y$. Applying Proposition~\ref{disjoint:condition:for:ssp} to $Y$, we conclude that $Y$ is \ssp.

If $X$ is \strong,
by Theorem~\ref{equiv}(ii),  
we can choose $x_n\in U_n$ for each $n\in\N$  
and find $x\in X$
such that 
$x\in \Cl_X(\{x_n:n\in \N\})\setminus \bigcup_{n\in\N}V_n$. 
Then 
$f(x_n)\in V_n$ for all $n\in\N$ 
and 
$f(x)\in\Cl_X(\{f(x_n):n\in \N\})\setminus \bigcup_{n\in\N}V_n$. 
Applying Theorem~\ref{equiv}(ii) to $Y$, we conclude that $Y$ is \strong.
\endproof

\begin{lemma}
\label{hitting:lemma}
Suppose that $X$ is a topological space having the following property:
For every countable family $\{U_n:n\in\N\}$ of non-empty open subsets of $X$,
there exists a 
\SSP\ 
subspace $Y$ of $X$ such that $U_n\cap Y\not=\emptyset$
for all $n\in\N$. Then $X$ is \SSP. 
\end{lemma}
\begin{proof}
Let $\{U_n:n\in\N\}$ be a sequence of pairwise disjoint non-empty open subsets of $X.$
By our hypothesis, 
$X$ contains a 
\SSP\ 
subspace $Y$ such that 
$V_n=U_n\cap Y\not=\emptyset$
for every $n\in\N$.
Since each $V_n$ is 
open in $Y$
and 
$Y$ is 
\SSP,
we can choose a point $x_n\in V_n\subseteq U_n$ for every $n\in\N$
such that the 
set $\{x_n:n\in\N\}$ is not closed in $Y$, and thus also in $X$ 
(the sequence $\{x_n:n\in\N\}$ has a subsequence converging
to some element of $Y$, respectively).
This shows that $X$ is \SSP.
\end{proof}

\begin{corollary}\label{dense}
If some dense subspace of a topological space
$X$ 
is \SSP,
then $X$ itself is \SSP.
\end{corollary}

\begin{proposition}\label{cofinal_are_ssp}
If every countable subset of a topological space $X$ is contained in a 
\SSP\ 
subspace of $X$,
then $X$ is 
\SSP.
\end{proposition}
\proof
Let $\{U_n:n\in\N\}$ be a family of non-empty open subsets of $X.$
Take $x_n\in U_n$ for every $n\in\N$.
By our hypothesis, the countable subset $\{x_n:n\in\N\}$ of $X$ is contained in a 
\SSP\ 
subspace $Y$ of $X$.
Since $x_n\in U_n\cap Y\not=\emptyset$ for every $n\in\N$, the conclusion follows from Lemma~\ref{hitting:lemma}.
\endproof

\begin{proposition}
A clopen subspace of a \SSP\ space is \SSP.
\end{proposition}

\proof
Let $Y$ be a clopen subspace of a \SSP\ space $X$.
Let $\{U_n:n\in\N\}$ be a sequence of non-empty open subsets of $Y.$ 
Since $Y$ is open in $X,$ each $U_n$ is also open in $X.$ 
Since $X$ is \SSP, 
we can use item (iii) of Theorem~\ref{equiv}
to 
choose a point $x_n\in U_n$ for every $n\in\N$
and $x\in X$ such that the set $\{n\in\N: x_n\in W\}$
is infinite for every open neighborhood $W$ of $x$ (respectively,
the sequence $\{x_n:n\in\N\}$ has a subsequence converging
to $x\in X$).
Since $\{x_n:n\in\N\}\subseteq Y$ and $Y$ is closed in $X$,  
it follows that $x\in Y$.
This shows that $Y$  is \SSP.
\endproof

\section{\SSPness\ in ($\Sigma$-)products}

\begin{lemma}\label{countable_product}
A countable product of \ssp\/ spaces is \ssp\/. 
\end{lemma}

\proof
Suppose that 
$X_i$ is a \ssp\/ space for every $i\in\N$, and let
$X=\prod_{i\in\N}X_i$ 
be the Tychonoff product 
of the sequence $\{X_i: i\in\N\}$. Let $\{U_n:n\in\N\}$ be a sequence of nonempty open subsets of $X$.
Without loss of generality, we may assume that each $U_n$ has the form $U_n=\prod_{i\in\N}U_{n,i}$, where $U_{n,i}$ is a nonempty open subset of $X_i$.

Let $J_{-1}=\N$. Since each $X_i$ is \ssp, 
by a straightforward induction on $i\in \N$, one can choose $x_{n,i}\in U_{n,i}$ for every $n\in \N$ and select a point $x_i\in X$ and an infinite set $J_i\subseteq J_{i-1}$ such that  the sequence
$\{x_{n,i}:n\in J_i\}$ converges to $x_i$.

Clearly,
$y_n=(x_{n,i})_{i\in\N}\in \prod_{i\in\N}U_{n,i}=U_n$ for every $n\in\N.$     

Let $j_0\in J_0$ be arbitrary.
For every $i\geq 1$, choose  $j_i\in J_i\setminus \{j_k:k\leq i-1\}$;
this can be done as $J_i$ is infinite.
By our construction, the set
$J_\omega =\{j_i:i\in\N\}$ is infinite,
while the set $J_\omega \setminus J_i$ is finite for every $i\in\N$.

Let $i\in \N$ be arbitrary. Since the sequence 
$\{x_{n,i}:n\in J_i\}$ converges to $x_i$ and $J_\omega \setminus J_i$ is finite, it follows that the sequence $\{x_{n,i}:n\in J_\omega\}$ converges to $x_i$ as well.
Since this holds for every $i\in\N$, we conclude that
the sequence $\{y_n:n\in J_\omega\}$ converges to 
$y=(x_{i})_{i\in\N}\in X.$
\endproof

\begin{definition}
If $\{X_i:i\in I\}$ is a family of sets, 
$X=\prod\{X_i:i\in I\}$ is their product
and $p$ is a point in $X$, 
then the subset 
\begin{equation}
\label{Sigma:product}
\Sigma  (p,X)=
\{f\in X:|\{i\in I:f(i)\not=p(i)\}|
\leq \omega  \}
\end{equation}
  of $X$ is called the {\em $\Sigma $-product of $\{X_i:i\in I\}$ with the basis point $p\in X.$\/} 
\end{definition}

The standard proof of the following lemma is omitted.
\begin{lemma}
Let $\{X_i:i\in I\}$ be a family of topological spaces and $X=\prod\{X_i:i\in I\}$
be its product and $p\in X$.
Then every countable subset of $\Sigma(p,X)$ is contained 
in a subspace of $\Sigma(p,X)$ homeomorphic to the product $\prod_{i\in J} X_i$ 
for some at most countable set $J\subseteq I$.
\end{lemma}

\begin{theorem}\label{sigma_product}
Let $\{X_i:i\in I\}$ be a family of topological spaces and $X=\prod\{X_i:i\in I\}$
be its product and $p\in X$. 
\begin{itemize}
\item[(i)] If all $X_i$ are \ssp, then so is  $\Sigma(p,X)$.
\item[(ii)] If $\prod_{i\in J} X_i$ is \ssp\ for every at most countable set $J\subseteq I$, then so is $\Sigma(p,X)$.
\end{itemize}
\end{theorem}
\proof
Let $C$ be a countable subset of $\Sigma(p,X)$. 
By \eqref{Sigma:product}, for every $c\in C$, the set $J_c=\{i\in I:f(i)\not=p(i)\}$ is at most countable. Since $C$ is countable, 
the set $J=\bigcup_{c\in C}J_c$ is at most countable as well. 
Note that $C\subseteq Y$, where $Y=\{f\in X: \{i\in I:f(i)\not=p(i)\}\subseteq J\}$ is a subset of $\Sigma(p,X)$ homeomorphic to $X_J=\prod\{X_i:i\in J\}$.

(i) Since $X_J$ is \ssp\ by Lemma~\ref{countable_product},
from
Proposition~\ref{cofinal_are_ssp} we conclude that $\Sigma(p,X)$ is \ssp.

(ii) Since $X_J$ is \strong\ by our assumption, 
applying
Proposition~\ref{cofinal_are_ssp} we conclude that $\Sigma(p,X)$ is \strong.
\endproof

\begin{corollary}\label{product}
A product of topological spaces is \ssp\/ if and only if each factor is \ssp\/. 
\end{corollary}
\proof
$(\Rightarrow)$ Follows from Proposition~\ref{continuous_image}.

$(\Leftarrow)$ Follows from~Theorem \ref{sigma_product}~(i) and Corollary~\ref{dense}, as each $\Sigma $-product
is dense in the product.
\endproof

\begin{corollary}
\label{dyadic:spaces}
Every dyadic space (in particular, every compact group) is \ssp.
\end{corollary}
\begin{proof}
Let $X$ be a dyadic space. Then $X$ is a continuous image of $\{0,1\}^\tau$ for a suitable cardinal $\tau$.
Since $\{0,1\}$ is obviously \ssp, Corollary~\ref{product} implies that $\{0,1\}^\tau$ is \ssp. Now \ssp ness of $X$ follows from Proposition~\ref{continuous_image}.

The ``in particular'' clause follows from the well-known fact that
compact groups are dyadic \cite[Theorem 4.1.7]{AT}.
\end{proof}

\begin{example}
The topological space $\beta \omega $ can be embedded in $D^\C,$ where $D=\{0,1\}$ is the two point discrete space.
By Corollary~\ref{dyadic:spaces}, $D^\C$ is \ssp\/. Since $\beta \omega $ is not \ssp\ by \cite[Example 2.9]{DPSW}, 
{\em a compact (and thus, closed) subspace of a \ssp\/ space is not necessarily \ssp\/}.
\end{example}

\begin{example}
Let $X$ be a countably compact space such that its square $X^{2}$ is not pseudocompact
\cite[Example 3.10.9]{E}. 
Since countably compact spaces are \strong\ and \strong\ spaces are pseudocompact (see Diagram 1),
this shows that  
{\em \strongness\ is not a productive property\/}. 
\end{example}

\section{\SSPness\ in topological groups}

One may expect that in the class of topological groups additional implications would hold that are not present for general topological spaces 
in Diagram 1.
The following result of this type was proved in \cite{AMPRT}:

\begin{fact}
Every pseudocompact group is \sp.
\end{fact}

In this section we establish another result of the same flavour.

\begin{definition}
A topological space $X$ is said to be {\em $\omega $-bounded\/} if the closure of any countable subset of $X$ is compact. 
\end{definition}

\begin{theorem}\label{omega_group_are_ssp}
Every $\omega $-bounded group is \ssp\/.
\end{theorem}
\proof
Suppose that $X$ is an $\omega $-bounded group. 
Let $S$ be a countable subset of $X.$ The subgroup $Z$ generated by $S$ is countable. Hence,
$Y=$ cl$_X(Z)$ is a compact group.  By Corollary~\ref{dyadic:spaces},
$Y$ is \ssp.  
By Proposition~\ref{cofinal_are_ssp},  $X$ is \ssp.
\endproof

In topological groups we have the following diagram. 
\begin{center} \medskip\hspace{1em}\xymatrix{
\text{compact}\ar[r]&  \omega \text{-bounded}\ar@{=>}^5[r] \ar[d]_6&\text{\ssp}\ar[d]_7\\
&\text{countably compact}\ar[r]^2&\text{\strong}\ar[d]_4\\
&\text{\sp}\ar@{<=>}[r] & \text{pseudocompact}
  }
\label{Diagram:2}
\end{center}
\begin{center}
Diagram 2.
\end{center}

In this diagram, double arrows denote implications that hold only in the class of topological groups and do not hold for general topological spaces.

\begin{example}
\label{ex:5.4}
By Corollary~\ref{dyadic:spaces}, the Cantor cube $D^\C$ is \ssp. 
By \cite[Theorem 3.10.33]{E}, $D^\C$ is not sequentially compact. 
Therefore, the compact group $G=D^\C$ is \ssp\ but not sequentially pseudocompact.
Hence, arrow 1 is not reversible even for compact (Abelian) groups.
\end{example}

\begin{example}\label{ex:1}
There is a pseudocompact group $G$ which is not \strong\
\cite[Example 2.4]{GT}. 
By \cite[Proposition 1.10]{AMPRT}, $G$ is \sp. Therefore,
arrows~3 and~4 of Diagram~1 are not reversible even for topological groups.
\end{example}

If $H$ is a pseudocompact group, then $H$ is a $G_\delta $-dense subgroup of its Raikov completion $\rho H.$
By \cite[Theorem 3.7.16]{AT}, $\rho H$ is compact and by Corollary~\ref{dyadic:spaces}, $\rho H$ is \ssp. Hence,
Example~\ref{ex:1} proves that \ssp ness is not hereditary with respect to $G_\delta $-dense subgroups.

\begin{example}
By the Hewitt-Marczewski-Pondiczery theorem, $D^\C$ is separable. Therefore, we can fix a countable dense subset $S$ of
$D^\C$. By a standard closing off argument, we can construct a 
countably compact subgroup $G$ of $D^\C$ such that $S\subseteq G$ and $|G|=\C$. Since $|D^\C|=2^\C>\C$, we conclude that $G$ is a proper subgroup of $D^\C$. Since $S$ is dense in $D^\C$, this implies that the closure of the countable set $S$ in $G$ is not compact. Thus, $G$ is not $\omega$-bounded.
This example shows that arrow~6 is not reversible even for topological groups.
\end{example}

\begin{example}
\label{CH:example}
Under CH, there is a countably compact group $G$ without convergent sequences. Then, $G$ is not \ssp. 
This example proves that, 
assuming CH, arrow~7 is
not reversible even for topological groups.  
\end{example}

We do not know if an additional set-theoretic assumption in this example can be omitted; see Question \ref{ZFC:question} below.

In Example~\ref{ex:old:5.6} below we shall show that arrow~2 of Diagram~1 and arrow~5 of Diagram~2 are not reversible in the class of topological groups.

\section{Building dense $\V$-independent sets in $\C$-powers, for a fixed variety $\V$}

We start 
this section 
with
the following general lemma.

 \begin{lemma}
\label{convergent_closed_sequences:0}
Let $H$ be a separable metric space.
Then for every uncountable subset $Y$ of $H$ one can find 
a set $X\subseteq Y$ such that
$X=\bigcup_{\alpha<|Y|} S_\alpha$, where:
\begin{itemize}
\item[(i)]  each $S_\alpha$ is a
countably infinite compact set with a single non-isolated point
(that is, $S_\alpha$ is a non-trivial convergent sequence taken together with its limit),
   and 
\item[(ii)]   
   $S_\alpha\cap S_\beta=\emptyset$ whenever $\alpha\not=\beta$.
\end{itemize}
\end{lemma}
\begin{proof}
First, note that every uncountable subset $Z$ of $H$ contains a 
countably infinite compact set $S$ with a single non-isolated point. Indeed, since $Z$ is a subspace of a separable metric space 
$H$, it has a countable base. On the other hand, $Z$ is uncountable, so it cannot be discrete. Therefore, $Z$ contains a non-isolated point $z$. Therefore, by induction on $n\in\N$, we can choose a point $x_n\in U(z,1/n)\setminus \{z, x_0,\dots,x_{n-1}\}$, where $U(z,1/n)$ is the ball centered at $z$ of radius $1/n$. Now $S=\{x_n:n\in\N\}\cup\{z\}$ is the desired set.

By transfinite induction on $\alpha<|Y|$, we shall select
a countably infinite subset $S_\alpha$ of $Y$ with a single non-isolated point such that $S_\alpha\cap \bigcup_{\beta<\alpha} S_\beta=\emptyset$.

For $\alpha=0$, we use the above observation (with $Z=Y$) to choose the required $S_0$.

Suppose now that the ordinal $\alpha$ satisfies $0<\alpha<|Y|$
and the family $\{S_\beta:\beta<\alpha\}$ has already been chosen.
Note that $|\bigcup_{\beta<\alpha} S_\beta|\le\alpha+\omega<|Y|$,
so the set $Z=Y\setminus \bigcup_{\beta<\alpha} S_\beta$
satisfies $|Z|=|Y|$; in particular, $Z$ is uncountable. Applying
the above observation to this $Z$, we get the desired $S_\alpha$. 
\end{proof}

A
{\em variety of groups} 
is
a class of groups closed under Cartesian products, subgroups and quotients \cite{Neu}.

\begin{definition}
Let $\V$ be a non-trivial variety of groups.
A subset $X$ of a group $G$ is said to be {\em $\V$-independent} if 
\begin{enumerate}
\item[(i)] $\gp{X}\in\V,$ and
\item[(ii)] for each map $f:X\to H\in\V$ there exists a unique homomorphism $\tilde{f}:\gp{X}\to H$ extending $f.$
\end{enumerate}
\end{definition}

We shall need the following useful fact
\cite[Lemmas 2.3 and 2.4]{DiS}.

\begin{lemma}\label{lemma2.3_DiS}
\label{lemma2.4_DiS}
Let $\V$ be a variety of groups and $X$ be a subset of a group $G$. 
\begin{itemize}
\item[(i)]
$X$   is $\V$-independent if and only if
each finite subset of $X$ is $\V$-independent. 
\item[(ii)] If $H$ is a group and
$f:G\to H$ is a homomorphism such that
$f(X)$ is a $\V$-independent subset of $H$, $\langle X\rangle\in\V$
and $f\restriction_X:X\to H$ is an injection, then $X$ is a $\V$-independent subset of $G$. 
\end{itemize}
\end{lemma}

\begin{definition}
Let $\V$ be a non-trivial variety of groups.
For a group $G$, the cardinal
	$$r_\V(G)=\sup\{|X|:X \textrm{ is a } \V \textrm{-independent subset of } G\}$$
	is  called 
	the {\em $\V$-rank} of $G$. 
\end{definition}

\begin{lemma}\label{convergent_closed_sequences}
Let $\V$ be a non-trivial variety. If $H$ is a separable metric group such that $r_\V(H)\geq \omega$, 
then there exists a $\V$-independent subset
$X\subseteq H^{\omega}$ such that  $X$
contains $\C$ many pairwise disjoint convergent closed sequences in $X$.
\end{lemma}

\proof
Since $H$ is a separable metric group, so is $H^\omega$.
By \cite[Lemma 4.1]{DiS},
the 
group $H^\omega$ has a $\V$-independent subset $Y$ of cardinality $\C$. Applying
Lemma~\ref{convergent_closed_sequences:0} to this $Y$, we can choose a subset $X$ of $Y$ as in the conclusion of this lemma.
Since $X$ is a subset of the $\V$-independent set $Y$, it is also 
$\V$-independent by Lemma~\ref{lemma2.3_DiS}~(i).
\endproof

\begin{theorem}\label{free_seq_pseudo}
Suppose that $H$ is a 
compact 
metric group with $r_\V(H)\geq \omega $. 
Then $H^\C$ contains a dense \ssp\ $\mathcal V$-independent subset of size continuum.
\end{theorem}
\proof
Define $L=H^\omega.$ 
By Lemma~\ref{convergent_closed_sequences},
there exists a $\mathcal V$-independent set 
$$
R=\{r_{\alpha ,n}:\alpha <\C,n\in\N\}\subseteq L
$$
such that 
\begin{equation}
\text{the sequence }
\{r_{\alpha ,n}:n\in\N\}
\text{ converges to }
r_{\alpha ,\omega}
\text{ for every }
\alpha<\C. 
\end{equation}

We shall call a subset $V$ of $L^\C$ a {\em basic open set\/} if $V=\prod_{\gamma<\C} V_\gamma$, where each $V_\gamma$ is an open subset of $L$ and the set $\mathrm{supp} (V)=\{\gamma<\C: V_\gamma\neq L\}$ is finite.
Let $\B$ be a base for $L^\C$ 
consisting of basic open subsets $L^\C$ such that $|\B|=\C$.
Let
$\mathcal U=[\B]^\omega.$ Since $|\mathcal U|=\C,$ 
we can enumerate 
$$ \mathcal U=\{\{U_{\alpha ,n}:n\in\N \}: \alpha <\C\},
$$
where $U_{\alpha,n}\in\B$ for all $\alpha<\C$ and $n\in\N$.

Let $\alpha <\C. $  
Since each $U_{\alpha ,n}$ is a basic open set in $L^\C$,
the set $\mathrm{supp}(U_{\alpha ,n})$ is finite, and so 
the set
\begin{equation}
\label{eq:S_alpha}
S_\alpha =\bigcup _{n\in\N }\mathrm{supp}(U_{\alpha ,n})
\end{equation}
is at most countable.
Let $p_\alpha:L^\C\to L^{S_\alpha}$ be the natural projection.

Fix $n\in\N$. Since $U_{\alpha ,n}$ is a non-empty basic open subset of $L^\C$ and $\mathrm{supp}(U_{\alpha ,n})\subseteq S_\alpha$ by \eqref{eq:S_alpha}, the set $W_n=p_\alpha(U_{\alpha ,n})$
is a non-empty open subset of $L^{S_\alpha}$ such that
\begin{equation}
\label{new:eq:7}
U_{\alpha ,n}=p_\alpha^{\leftarrow}(W_n).
\end{equation}

Since $S_\alpha$ is countable and $L$ is a compact metric space,
the space $L^{S_\alpha}$ is compact and metric. In particular,
$L^{S_\alpha}$ is sequentially compact, and so also 
\ssp; see Diagram 1.
Therefore,
we can choose a point $x_{\alpha,n}\in W_n$ for every $n\in \N$
in such a way that a suitable subsequence $\{x_{\alpha,n}:n\in J_\alpha\}$ of the sequence $\{x_{\alpha,n}:n\in\N\}$
converges to some element $x_{\alpha,\omega}\in L^{S_\alpha}$.

For every $n\in\omega+1$ ,
define $y_{\alpha,n}\in L^\C$ by 
\begin{equation}
\label{eq:5:old}
y_{\alpha ,n}(\beta)=
\left\{\begin{array}{ll}
x_{\alpha ,n}(\beta) & \mbox{if }\beta\in S_\alpha \\
r_{\alpha ,n}& \mbox{if }\beta\in \C\setminus S _\alpha
\end{array}
\right.
\ \ \ 
\mbox{for every }
\beta<\C.
\end{equation}

\begin{claim}
\label{claim:1}
For every $\alpha<\C$, the following holds:
\begin{itemize}
\item[(i)] $y_{\alpha ,n}\in U_{\alpha ,n}$ for every $n\in\N$;
\item[(ii)] the sequence $\{y_{\alpha ,n}:n\in J_\alpha \}$ converges to $y_{\alpha ,\omega }.$ 
\end{itemize}
\end{claim}
\begin{proof}
Fix an arbitrary $\alpha<\C$.

(i)
By \eqref{eq:5:old}, 
$p_\alpha(y_{\alpha,n})=x_{\alpha,n}\in W_n$, so
$
y_{\alpha ,n}\in p_\alpha^{\leftarrow}(W_n)=U_{\alpha ,n}
$
by \eqref{new:eq:7}.

(ii) It suffices to show that the sequence 
$\{y_{\alpha ,n}(\beta):n\in J_\alpha \}$ converges to $y_{\alpha ,\omega }(\beta)$
for every $\beta<\C$. 
In checking this,we consider two cases.

\smallskip
{\em Case 1\/}.  By our choice of the set $J_\alpha$, 
the sequence $\{x_{\alpha,n}:n\in J_\alpha\}$
converges to $x_{\alpha,\omega}\in L^{S_\alpha}$ in $L^{S_\alpha}$.
Since $\beta\in S_\alpha$,
this implies that the sequence 
$\{y_{\alpha ,n}(\beta):n\in J_\alpha \}$ converges to $y_{\alpha ,\omega }(\beta)$.

\smallskip
{\em Case 2\/}. $\beta\in \C\setminus S_\alpha$.
By \eqref{eq:5:old}, 
$y_{\alpha,n}(\beta)=r_{\alpha,n}$ for every $n\in\omega+1$. 
Now it remains only to note that, by our construction,
the sequence $\{r_{\alpha,n}:n\in\N\}$ converges to $r_{\alpha,\omega}$.
\end{proof}
\begin{claim}
\label{claim:2}
$Y=\{y_{\alpha ,n}:\alpha<\C ,n\in \omega+1 \}\subseteq L^\C$ is \ssp\
and dense in $L^\C$.
\end{claim}
\begin{proof}
Let
$\A=\{A_n:n\in\N\}$ be a countable sequence of elements of $\B. $ There exists $\alpha <\C$ such that 
$A_n=U_{\alpha ,n}$ for every $n\in\N$; that is,
$\A=\{U_{\alpha ,n}:n\in\N \}$.
By Claim~\ref{claim:1},
$y_{\alpha ,n}\in U_{\alpha ,n}$ for each $n\in \N $ and the sequence $\{y_{\alpha ,n}:n\in J_\alpha  \}$
converges to $y_{\alpha ,\omega }$. 
Since $\{y_{\alpha,n}:n\in\omega+1\}\subseteq Y$,
this proves that $Y$ is \ssp.

Let $O$ be a non-empty open subset of $L^\C$.
Since $\B$ is a base of $L^\C$, there exists $B\in\B$ such that
$B\subseteq O$. Let $\A$ be the constant sequence
all elements of which are equal to $B$. 
Then $\A=\{U_{\alpha ,n}:n\in\N \}$ for some  $\alpha<\C$,
and so $\{y_{\alpha,n}:n\in J_\alpha\}\subseteq B\cap Y\subseteq O\cap Y$.
This shows that $Y$ 
is dense in $L^\C.$
\end{proof}

\begin{claim}
\label{claim:3}
$Y$ 
is $\mathcal V$-independent.
\end{claim}
\begin{proof}
By Lemma~\ref{lemma2.3_DiS}~(i), it suffices to show that
every finite subset $X$ of $Y$ is $\V$-independent. Let $X$ be a finite subset of $Y$. There
exists a
finite subset $F$ of $\C$ 
such that 
\begin{equation}
\label{eq:X}
X\subseteq \{y_{\alpha,n}:(\alpha,n)\in F\times (\omega+1)\}.
\end{equation}

Since each set $S_\alpha$ is at most countable and $F$ is finite, 
the set $S=\bigcup_{\alpha\in F}S_\alpha$
is at most countable.
Therefore,
we can choose $\beta\in \C\setminus S.$
Let $\pi_\beta:L^\C\to L$ be the projection on the $\beta$th coordinate.
To prove that $X$ is $\V$-independent, it suffices to check 
that $G=L^\C$, $H=L$, $X$ and $f=\pi_\beta$ satisfy the assumptions of  Lemma~\ref{lemma2.3_DiS}~(ii).

Since $\V$ is a variety of groups, it is closed under taking arbitrary products and subgroups.
Since
$L\in\V$ and $\langle X\rangle$ is a subgroup of $L^\C$, this implies $\langle X\rangle\in\V$.
 
Let $(\alpha,n)\in F\times (\omega+1)$ be arbitrary.
Then 
$\alpha\in F$, and so $S_\alpha\subseteq S$.
Since $\beta\not\in S$, this implies
$\beta\not \in S_\alpha$.
Therefore,
$
\pi_\beta(y_{\alpha,n})=r_{\alpha,n}
$
by \eqref{eq:5:old}.
This shows that
$$\{\pi_\beta(y_{\alpha,n}):(\alpha,n)\in F\times (\omega+1)\}
=
\{r_{\alpha,n}:(\alpha,n)\in F\times (\omega+1)\}.
$$
Since the latter set is a faithfully indexed subset of 
$R$, it follows from \eqref{eq:X} that 
$f\restriction_X=\pi_\beta\restriction_X:X\to R$ 
is an  injection.
Since $f(X)=\pi_\beta(X)\subseteq R$
and $R$
is $\V$-independent,
$f(X)$ is $\V$-independent
 by Lemma~\ref{lemma2.3_DiS}~(i).
\end{proof}

The conclusion of our theorem now follows from Claims~
\ref{claim:1},~\ref{claim:2} and~\ref{claim:3}.
\endproof

\section{\ssp\ group topologies on $\V$-free groups}

\begin{definition}
\label{def:V:base:V:free}
Let $\V$ be a non-trivial variety of groups.
\begin{enumerate}
\item[(i)] A subset $X$ of a group $G$ is a {\em $\V$-base} of $G$ if $X$ is $\V$-independent and  $\gp{X}=G.$
\item[(ii)] A group is {\em $\V$-free} if it contains a $\V$-base.   
\end{enumerate}
\end{definition}

If $\V$ is a non-trivial variety, then for every set $X$ there exists a unique (up to isomorphism) group $F_X(\V)$ such that 
$X\subseteq F_X(\V)$ and $X$ is a $\V$-base of $F_X(\V).$ 

\begin{definition}
For every cardinal $\tau $, we use $F_\tau (\V)$ to denote the unique (up to isomorphism) $\V$-free group having a $\V$-base of 
cardinality $\tau .$ 
\end{definition}

\begin{definition}\cite[Definition 5.2]{DiS}
A variety $\V$ is said to be \emph{precompact\/} if there exists a compact zero-dimensional metric group $H\in\V$ with $r_\V(H)\geq \omega .$  
\end{definition}

\begin{corollary}
\label{precompact:corollary}
For a precompact variety $\V$, the group $F_\C(\V)$ admits a \ssp\ zero-dimensional group topology. 
\end{corollary}

\proof
By Theorem~\ref{free_seq_pseudo},  $H^\C$ contains a dense \ssp\ $\mathcal V$-independent subset 
$X$ of size continuum. 
It follows from Definition~\ref{def:V:base:V:free}~(i)
that $X$ is a $\V$-base of $\gp{X}$, so $\gp{X}$ is algebraically isomorphic to $F_\C(\V)$.
Since 
$X\subseteq \gp{X}\subseteq H^\C$ and $X$ is dense in 
$H^\C$, we conclude that $X$ is also dense  in $\gp{X}$.
Since $X$ is \ssp, from Corollary~\ref{dense}
we conclude that $\gp{X}$ is \ssp\ as well.
Since $H$ is zero-dimensional, so is $H^\C$ and $\gp{X}$. 
\endproof

Since both the variety  $\mathcal{G}$ of all groups and the variety 
$\A$ of all Abelian groups are
precompact by \cite[Lemma 5.3]{DiS}, 
from Corollary~\ref{precompact:corollary} we obtain the following
corollary.
\begin{corollary}
\label{G:A:free:groups}
Both the free group $F_\C(\mathcal{G})$ with $\C$-many generators and the free Abelian group $F_\C(\A)$ with $\C$-many generators admit a \ssp\ zero-dimensional group topology. 
\end{corollary}

\begin{corollary}\label{free_abelian_seq_pseudo}
The group $\T^\C$ contains the free Abelian group $F_\C(\A)$ with $\C$-many generators  as a dense \ssp\
subgroup. 
\end{corollary}
\begin{proof}
The circle group $\T$ is a compact 
metric Abelian group. By \cite[Lemma 7.1.6]{AT}, $\T$ contains
an 
$\A$-independent set $X$ of cardinality $\C.$ By 
Theorem~\ref{free_seq_pseudo},
the group $\T^\C$ contains 
a dense $\A$-independent \ssp\ subspace $X$ of size $\C$. 
Arguing as in the proof of Corollary~\ref{precompact:corollary}, we deduce that
$\gp{X}$ is a dense \ssp\ subgroup of $\T^\C$ algebraically isomorphic to $F_\C(\A)$.
\end{proof}

\begin{example}
Let $G=F_\C(\mathcal{G})$ be the $\mathcal{G}$-free group of size $\mathfrak{c}$.
By Corollary~\ref{G:A:free:groups}, {\em $G$ admits a \ssp\ group topology}. 
On the other hand, {\em no non-trivial subgroup of $G$ admits a countably compact group topology}.
Indeed, if $H$ is a subgroup of $G$, then it is $\mathcal{G}$-free, and it is known that no non-trivial $\mathcal{G}$-free group 
admits a countably compact group topology. 
In particular, {\em there exists a \ssp\ group without non-trivial countably compact subgroups}.
\end{example}

\begin{example}\label{ex:old:5.6}
Let $G=F_{\C}(\mathcal{A})^\omega$ be the countable power of the $\mathcal{A}$-free group of size $\mathfrak{c}$.
Then $G$ is an Abelian group of size $\mathfrak{c}$.
By Corollary~\ref{G:A:free:groups}, $F_{\C}(\mathcal{A})$ admits a \ssp\ group topology. Since the class of \ssp\ spaces is closed under arbitrary products, $G$ also admits a \ssp\ group topology.
By \cite[Theorem 17]{Tomita}, this topology cannot be countably compact.
Therefore, {\em $G$ is an Abelian group (of size $\mathfrak{c}$) which admits a \ssp\ group topology, yet does not admit a countably compact group topology.} This example proves that arrows 2 and 5 are not reversible even for topological groups.  
\end{example}

\section{Final remarks and open questions}

Garc\'ia-Ferreira and Tomita asked whether every compact group 
contains a proper dense \strong\ subgroup
\cite[Question 2.10]{GT}.
This question
has a 
consistent positive answer. Indeed, assuming $2^{\omega _1}>\C$,
it was proved in \cite{IS} 
that  every non-metrizable
compact group has a proper dense countably compact subgroup,
and countably compact spaces are \strong; see  Diagram 1.

In ZFC, one can obtain a positive solution to even stronger version of the question of Garc\'ia-Ferreira and Tomita for large classes of topological groups.
 
\begin{remark}
(i) It was proved in \cite{IS} that {\em if a compact group $G$ can be mapped onto a product of $w(G)$-many nontrivial 
compact metric groups,
then $G$ contains $|G|$-many dense, 
$\omega $-bounded proper subgroups.\/} 
(Here $w(G)$ denotes the weight of $G$.)

(ii) The assumption of item (i) is satisfied when $G$ is either connected or Abelian 
\cite{IS}. Therefore,
{\em if a non-metrizable compact 
group $G$ is either connected or Abelian,
then $G$  contains $|G|$-many dense, 
$\omega $-bounded proper subgroups.\/} 

(iii) Since $\omega$-bounded groups are \ssp\ 
by Corollary~\ref{dyadic:spaces}, it follows from (ii) that
{\em if a non-metrizable compact 
group $G$ is either connected or Abelian,
then $G$ 
contains $|G|$-many dense, 
\ssp\/ proper subgroups.\/}
\end{remark}

The following question is related to this remark.

\begin{question}
Does every non-metrizable \SSP\ Abelian group contains a proper dense \SSP\ subgroup?
\end{question}

Example \ref{CH:example}
justifies the following question.

\begin{question}
\label{ZFC:question}
(i) Is there a ZFC example of \sp\ (Abelian) group which is not \ssp?

(ii) Is there a ZFC example of a countably compact  (Abelian) group which is not \ssp?
\end{question}

It should be interesting to investigate how the new notion is behaving in $C_p(X,G)$-theory \cite{SS}.
Let $G$ be a topological group. For a topological space $X$, we denote by $C_p(X,G)$ the group of all $G$-valued continuous maps from $X$ into $G$ endowed with the topology of pointwise convergence, i.e., the subspace topology $C(X,G)$ inherits from the Tychonoff product $X^G$. 
A space $X$ is said to be {\em $b_G$-discrete\/} if every countable subset $D$ of $X$ is discrete and each function from $D$ to $G$ can be extended to a continuous function from $X$ to $G$.

\begin{remark}
{\em If $G$ is a 
metric
group and $X$ is a topological space such that $C_p(X,G)$ is pseudocompact and dense in $G^X$, then $G$ is compact
 and $X$ is $b_G$-discrete.\/}
Indeed, 
since
$C_p(X,G)$ is pseudocompact and dense in $G^X$, the latter space is pseudocompact. Since $G$ is a continuous projection of the pseudocompact space $G^X$, it is pseudocompact. 
Now it remains only to note that a metric
pseudocompact space is compact.
As a dense pseudocompact subgroup of the compact group $G^X$, the group $C_p(X,G)$ is $G_\delta$-dense in $G^X$.
Since $G$ is first countable, 
this is equivalent to $X$ 
being
$b_G$-discrete
\cite[Proposition 4.6]{DRT}. 
\end{remark}

This remark motivates the compactness condition on $G$ in the following problem.

\begin{problem}
\label{8.5}
Let $G$ be a compact metric group and $X$ be a $b_G$-space. 
Find necessary and sufficient conditions on a space $X$ for $C_p(X,G)$ to be \SSP.
\end{problem}

Let $G$ be a topological group.
Recall that spaces $X$ and $Y$ are called {\em $G$-equivalent\/}
if topological groups $C_p(X,G)$ and $C_p(Y,G)$ are topologically isomorphic \cite{SS}.

\begin{question}
\label{8.6}
Let $G$ be a topological group, and let
$X$ and $Y$ be $G$-equivalent spaces such that $C_p(X,G)$ is dense in $G^X$ and $C_p(Y,G)$ is dense in $G^Y$.
If $X$ is \SSP, must $Y$ also be \SSP?
\end{question}

\medskip
\noindent
{\bf Acknowledgement:\/} 
The authors would like to thank Boaz Tsaban for his comment
that inspired Problem \ref{8.5} and Question \ref{8.6}.
This paper was written during the first listed author's stay at the Department of Mathematics 
of Faculty of Science of Ehime University (Matsuyama, Japan)
in the capacity of Visiting Foreign Researcher under the support by 
CONACyT of M\'exico: Estancias Posdoctorales al Extranjero propuesta No. 263464. 
He would like to thank CONACyT for its support and the host institution for its hospitality.


\begin{thebibliography}{99} 
\bibitem{AG} A. V. Arkhangel'skii, H. M. M. Genedi,
\textit{Properties of position type: relative strong pseudocompactness,} Proc. Steklov Inst. Math. 3 (193) (1993) 25-27.

\bibitem{AT} A. Arhangel'skii and M. Tkachenko.
	\textit{Topological Groups and Related Structures}.
	Atlantis Studies in Mathematics, Editor: J. van Mill, Atlantis Press/World Scientific, 2008.

\bibitem{AMPRT} G. Artico, U. Marconi, J. Pelant, L. Rotter, M. Tkachenko, 
\textit{Selections and suborderability}, Fund. Math. 175 (2002), 1-33. 
	
\bibitem{CR} W. W. Comfort, K. A. Ross, 
  \textit{Pseudocompactness and uniform continuity in topological groups},
Pacific J. Math. 100 (1982), 61-84.
			
\bibitem{Di}D. N. Dikranjan, 
\textit{Zero-dimensionality of some pseudocompact groups}, Proc. Am. Math. Soc. 120(4) (1994) 1299-1308.

\bibitem{DiS} D. N. Dikranjan, D. B. Shakhmatov, 
\textit{Algebraic structure of pseudocompact groups},
Mem. Amer. Math. Soc. 133/633 (1998), 83 pages. 

\bibitem{DRT} A. Dorantes-Aldama, R. Rojas-Hern\'andez and \'A. Tamariz-Mascar\'ua, {\it Weak pseudocompactness  
on spaces of continuous functions}, Topology Appl. 196 Part A (2015), 72-91.

\bibitem{DPSW} A. Dow, J. R. Porter, R. M. Stephenson, R. G. Woods, 
\textit{Spaces whose pseudocompact subspaces are closed subsets},
Appl. Gen. Topol. 5 (2004), 243-264. 

\bibitem{E} R. Engelking, 
  \textit{General Topology},
Sigma Series in Pure Mathematics, Heldermann, Berlin, 1989.

\bibitem{F} Z. Frol\'ik, 
  \textit{The topological product of two pseudocompact spaces},
Czechoslovak Math. J. 10 (85) (1960), 339-349. 

\bibitem{GO} S. Garc\'ia-Ferreira, Y. F. Ortiz-Castillo,
  \textit{Strong pseudocompact properties},
Comment. Math. Univ. Carolin. 55,1 (2014) 101-109.

\bibitem{GT} S. Garc\'ia-Ferreira, A. H. Tomita,
  \textit{A pseudocompact group which is not strongly pseudocompact},
Topology Appl. 192 (2015) 138-144.

\bibitem{H} E. Hewitt, \textit{Rings of continuous functions I}, Trans. Amer. Math. Soc. 
	48 (1948), 45-99.

\bibitem{IS} G. Itzkowitz, D. B. Shakhmatov, 
 \textit{Dense countably compact subgroups of compact groups},
 Math. Japon. 45, no. 3 (1997), 497-501.

\bibitem{L} P. Lipparini, \textit{Products of sequentially pseudocompact spaces}, 
Appl. Gen. Topology 17, no. 1 (2016), 1--5.

\bibitem{Neu}
H. Neumann, \textit{Varieties of groups}, Springer-Verlag, Berlin-Heidelberg-New York, 1967. 

\bibitem{S} D. B. Shakhmatov,	
	\textit{A pseudocompact Tychonoff space all countable subsets of which are closed
	and $C^*$-embedded}, Topology Appl. 22 (1986), 139-144.

\bibitem{SS} D. Shakhmatov and J. Sp\v{e}v\'ak, \textit{Group valued continuous functions with the topology
of pointwise convergence}, 
Topology Appl. 157 (2010), 1518-1540.

\bibitem{Tomita} A. H. Tomita, \textit{The existence of initially $\omega_1$-compact group topologies on free Abelian groups is independent of ZFC}, Comment. Math. Univ. Carolin.
39, 2 (1998), 401-413. 
\end{thebibliography}
\end{document}